\documentclass[finalversion]{FPSAC2025}
\articlenumber{1}

\newtheorem{thm}{Theorem}[section]

\newtheorem{conj}[thm]{Conjecture}

\usepackage{lipsum}

\title[Conjectures on the Reduced Kronecker Coefficients]{Conjectures on the Reduced Kronecker Coefficients}

\author[Tao Gui]{Tao Gui\thanks{\href{mailto:guitao18@mails.ucas.ac.cn}{guitao18@mails.ucas.ac.cn}. }\addressmark{1}}

\address{\addressmark{1}Beijing International Center for Mathematical Research, Peking University, No.\ 5 Yiheyuan Road, Haidian District, Beijing 100871, P.R. China}

\received{October 01, 2024}

\abstract{We formulate a series of conjectures on the stable tensor product of irreducible representations of symmetric groups, which are closely related to the reduced Kronecker coefficients. These conjectures are certain generalizations of Okounkov's conjecture on the log-concavity of the Littlewood--Richardson coefficients and the Schur log-concavity theorem of Lam--Postnikov--Pylyavskyy. We prove our conjectures in some special cases and discuss some implications of these conjectures.}

\keywords{log-concavity, tensor products, representations of symmetric groups, reduced Kronecker coefficients, Schur polynomials}

\usepackage[backend=bibtex, giveninits]{biblatex}
\begin{document}

\maketitle

\section{Introduction}

The main purpose of this article is to announce and provide supporting evidence for some conjectures on the stable tensor product of irreducible representations of symmetric groups, which are closely related to the reduced Kronecker coefficients.

Recall that a sequence of real numbers $a_{0}, a_{1}, \ldots$ is called \emph{log-concave} if 
$$
a_{i}^{2} \geq a_{i-1} a_{i+1} \text { for all $i\geq 1$}.
$$
Log-concave sequences are very common in algebra, geometry, and combinatorics. In addition, many log-concave phenomena appear in representation theory. 

In an influential article \cite{okounkov2003would}, based on heuristics and analogies of physical principles, Okounkov made a remarkable conjecture (see \cref{Oconj}) that the Littlewood--Richardson coefficients $c_{\lambda\mu}^{\nu}$ are log-concave in $(\lambda, \mu, \nu)$. Although Okounkov's conjecture is false in general \cite{chindris2007counterexamples}, many consequences and interesting special cases are true. In particular, Okounkov's conjecture implies that tensor products of finite-dimensional irreducible polynomial representations of the general linear group are log-concave, which is proved in  \cite{lam2007schur} and called Schur log-concavity. 

Motivated by the Schur--Weyl duality, we would like to consider the corresponding conjectures for the symmetric groups, that is, by replacing the Littlewood--Richardson coefficients in Okounkov's conjecture with the Kronecker coefficients. It turns out the naive analogs for Kronecker coefficients are false, but it seems that certain log-concavity properties (see conjectures \ref{Gconj-1}, \ref{Gconj-2}, \ref{Gconj-3}, \ref{Gconj-4} and theorems \ref{thm-dim}, \ref{thm-tworow}, \ref{thm-hook}) reappears for the stable tensor product of irreducible representations of symmetric groups, whose structure constants are the reduced Kronecker coefficients.

The remaining part of this article is organized as follows. 
In \cref{Sec-Okounkov}, we recall Okounkov's conjecture on the log-concavity of the Littlewood--Richardson coefficients and some interesting implications and known special cases. In \cref{Sec-Kronecker}, we recall the Kronecker coefficients and discuss the convexity property of the Kronecker coefficients. In \cref{Sec-reduced}, we recall the reduced Kronecker coefficients. In \cref{Sec-Gconj}, we state our conjectures and evidence on the reduced Kronecker coefficients and some implications.

\section{Okounkov conjecture on the Littlewood--Richardson coefficients} \label{Sec-Okounkov}

Recall that the \emph{Littlewood--Richardson coefficients} $c_{\lambda \mu}^{\nu}$ are the structure constants of the tensor product of irreducible polynomial representations of general linear group $GL_{n}(\mathbb{C)}$:
$$
V(\lambda) \otimes V(\mu)=\bigoplus_\nu c_{\lambda \mu}^\nu V(\nu),
$$
where $\lambda, \mu$, and $\nu$ are partitions with lengths less than or equal to $n$. Okounkov made the following remarkable conjecture.

\begin{conj}(Disproved; Okounkov conjecture, see \cite[Conjecture 1]{okounkov2003would}) \label{Oconj}

The function $$
(\lambda, \mu, \nu) \rightarrow \log c_{\lambda \mu}^{\nu}
$$ 
is concave. That is, suppose $\left(\lambda_i, \mu_i, \nu_i\right), i=1,2,3$, are partitions such that
$$
\left(\lambda_2, \mu_2, \nu_2\right)=\frac{1}{2}\left(\lambda_1, \mu_1, \nu_1\right)+\frac{1}{2}\left(\lambda_3, \mu_3, \nu_3\right),
$$
then we have 
$$(c_{\lambda_2 \mu_2}^{\nu_2})^2 \geq c_{\lambda_1 \mu_1}^{\nu_1} c_{\lambda_3 \mu_3}^{\nu_3}.$$
\end{conj}

Okounkov's conjecture \ref{Oconj} is a very strong statement, which holds in the ``classical limit''
(see \cite[Section 3]{okounkov2003would}), but it is refuted in general in \cite{chindris2007counterexamples}. To describe the counterexamples, we use the multiplicity/exponential notation for a partition $(\lambda_1^{m_1}, \lambda_2^{m_2}, \lambda_3^{m_3} \cdots$, where $m_1$ is the number of $\lambda_1$ 's, $m_2$ is the number of $\lambda_2$ 's, etc. 

\begin{thm}[{\cite[Theorem 1.2]{chindris2007counterexamples}}] Let $n \geqslant 1$ be an integer and let $\lambda(n), \mu(n)$ be two partitions defined by
$$
\lambda(n)=\left(4^n, 3^{2 n}, 2^n\right) \quad \text { and } \quad \mu(n)=\left(3^n, 2^n, 1^n\right) .
$$
Then
$$
c_{\mu(n), \mu(n)}^{\lambda(n)}=\left(\begin{array}{c}
n+2 \\
2
\end{array}\right) \quad \text { and } \quad c_{2 \mu(n), 2 \mu(n)}^{2 \lambda(n)}=\left(\begin{array}{c}
n+5 \\
5
\end{array}\right) .
$$
Consequently, when $n \geqslant 21$, \cref{Oconj} fails for $\lambda_1=2\lambda(n)$, $\mu_1=\nu_1=2\mu(n)$, $\lambda_2=\lambda(n)$, $\mu_2=\nu_2=\mu(n)$, $\lambda_3=\mu_3=\nu_3=0$.
\end{thm}

However, several interesting implications and special cases of \cref{Oconj} are true.

First, as Okounkov observed in \cite[Section 2.6]{okounkov2003would}, concavity of $\log c_{\lambda \mu}^{\nu}$ implies that 
$$
\operatorname{supp} c_{\lambda \mu}^{\nu}=\left\{(\lambda, \mu, \nu), c_{\lambda \mu}^{\nu} \neq 0\right\}
$$
is convex. In particular, since it contains the origin $(0,0,0)$, it is saturated. This shows that \cref{Oconj} implies the \emph{saturation property} of Littlewood--Richardson coefficients\footnote{In fact, since $c_{0,0}^{0}=1$, \cref{Oconj} implies that
$c_{k \lambda, k \mu, k \nu} \leq \left(c_{\lambda, \mu, \nu}\right)^k.
$}:
\begin{equation} \label{satu}
c_{k \lambda, k \mu}^{k \nu} \neq 0 \text { for some } \mathrm{k} \geq 1 \Rightarrow c_{\lambda \mu}^{\nu} \neq 0,
\end{equation}
which was established before Okounkov's conjecture by A. Knutson and T. Tao in \cite{knutson1999honeycomb} using the honeycomb model of $GL_{n}(\mathbb{C})$ tensor products. 

Note that Knutson and Tao’s proof of the Saturation Conjecture implies that the decision problem “whether $c_{\mu \nu}^\lambda>0$” is in P; as a comparison, the famous Littlewood--Richardson rule,
which gives a positive combinatorial interpretation for the Littlewood--Richardson coefficients $c_{\lambda \mu}^{\nu}$, shows that “computing $c_{\lambda \mu}^{\nu}$” is in \#P (counting problems associated with the decision problems in the set NP); in fact, it is \#P-complete, see \cite{narayanan2006complexity}.

Another interesting implication of Okounkov's conjecture is also already observed by Okounkov in \cite[Section 2.5]{okounkov2003would}. \cref{Oconj} would have implied that for all $\nu$,
\begin{equation} \label{two}
c_{\frac{\lambda+\mu}{2} \frac{\lambda+\mu}{2}}^{\nu} \geq c_{\lambda \mu}^{\nu},
\end{equation}
provided $\frac{\lambda+\mu}{2}$ is an integral weight (a.k.a., a partition). It is equivalent to the inclusion of representations
\begin{equation} \label{tensor}
V(\lambda) \otimes V(\mu) \subset V\left(\frac{\lambda+\mu}{2}\right)^{\otimes 2},
\end{equation}
which can be interpreted as saying that the representation valued function 
\begin{equation}
V: \lambda \mapsto V(\lambda)
\end{equation}
is concave with respect to the natural ordering and tensor multiplication of representations. Since Schur polynomials are the characters of the corresponding irreducible polynomial representations of $\operatorname{GL}_{n}(\mathbb{C)}$, this remarkable implication is called \emph{Schur log-concavity} and has been established by T. Lam, A. Postnikov, and P. Pylyavskyy.

\begin{thm}[{\cite[Theorem 12]{lam2007schur}}, weak version]
    \label{Thm-Schur} 
For two partitions $\lambda$ and $\mu$, suppose $\lambda+\mu$ has only even parts and let $s_{\lambda}$, $s_{\mu}$, and $s_{\frac{\lambda+\mu}{2}}$ be the corresponding Schur polynomials, then
$s_{\frac{\lambda+\mu}{2}}^2 - s_{\lambda} s_{\mu}$ is a non-negative linear combination of Schur polynomials.
\end{thm}

Last but not least, a special but interesting case of \cref{Oconj} is recently obtained in \cite{huh2022logarithmic}. The \emph{Kostka number} $K_{\lambda \mu}$---the coefficient of monomials $x^\mu$ in the Schur polynomial $s_{\lambda}$, also known as the \emph{weight multiplicity} $\operatorname{dim} \mathrm{V}(\lambda)_\mu$ of the Schur module $\mathrm{V}(\lambda)$---is a special case of the Littlewood--Richardson coefficients: we have 
$$K_{\lambda \mu}=c_{\kappa \lambda}^\nu,$$
where $\nu$ and $\kappa$ are the partitions given by $\nu_i=\sum_{j=i}^n \mu_j$ and $\kappa_i=\sum_{j=i+1}^n \mu_j$. One of the main results in \cite{huh2022logarithmic} states that the Kostka number $K_{\lambda \mu}$ is log-concave along the \emph{root directions}:
let $e_i$ be the $i$-th standard unit vector in $\mathbb{N}^m$, for $\mu \in \mathbb{Z}^m$ and distinct $i, j \in[m]$, set
$$
\mu(i, j)=\mu+e_i-e_j,
$$
then the sequence of weight multiplicities of $\mathrm{V}(\lambda)$ we encounter is always log-concave if we walk in the weight diagram along any root direction $e_i-e_j$.

\begin{thm}[{\cite[Theorem 2]{huh2022logarithmic}}] \label{Kostka}
For any partition $\lambda$ and any $\mu \in \mathbb{N}^m$, we have
$$
K_{\lambda \mu}^2 \geqslant K_{\lambda \mu(i, j)} K_{\lambda \mu(j, i)} \text { for any } i, j \in[m] \text {. }
$$
\end{thm} 

\section{Kronecker Coefficients and Their Convexity Property} \label{Sec-Kronecker}

Recall that the Kronecker coefficients $g_{\lambda \mu}^{\nu}$ are the structure constants of the tensor product (Kronecker product) of irreducible representations of the symmetric group $S_{d}$:
$$
V_{\lambda} \otimes V_{\mu}=\bigoplus_\nu g_{\lambda \mu}^\nu V_{\nu},
$$
where $\lambda, \mu$ and $\nu$ are partitions of $d$. They were introduced by Murnaghan in 1938 and they play an important role algebraic combinatorics and geometric complexity theory. 

By the representation theory of finite groups, these coefficients can be computed as
$$
g_{\lambda \mu}^\nu=\frac{1}{n !} \sum_{\sigma \in S_d} \chi^\lambda(\sigma) \chi^\mu(\sigma) \chi^\nu(\sigma),
$$
where $\chi^\lambda(\sigma)$ is the character value of the irreducible representation corresponding to partition $\lambda$ on a permutation $\sigma \in S_d$.
Since irreducible representations of the symmetric group $S_{d}$ have integral character values, the Kronecker coefficient $g_{\lambda \mu}^\nu$ is invariant under permutations of the three partitions. This should be compared with Littlewood--Richardson coefficients $c_{\lambda \mu}^\nu$, where it is only invariant under transposition of $\lambda$ and $\mu$.

The Kronecker coefficients are very different beasts from their cousins Littlewood--Richardson coefficients. For example, computing Kronecker coefficients is \#P-hard and contained in GapP \cite{burgisser2008complexity}. A recent work \cite{ikenmeyer2017vanishing} shows that the decision problem ``whether $g_{\mu \nu}^\lambda>0$'' is NP-hard. They lack ``nice'' formulas and what we can hope is to understand their asymptotic behavior in various regimes and inequalities they could satisfy. Finding a combinatorial interpretation for them has been described by Richard Stanley
as ``one of the main problems in the combinatorial representation theory of the
symmetric group''. 

Let us now only focus on one particular aspect---the convexity property of the Kronecker coefficient. The verbatim translation of the saturation property \eqref{satu} that
holds for the Littlewood--Richardson coefficients is known to be false for the
Kronecker coefficients \cite{briand2009reduced}. The simplest counterexample might be
$g_{(1,1)(1,1)}^{(1,1)}=0$ but $g_{(2,2)(2,2)}^{(2,2)}=1$. Indeed,
$$
g_{(N, N),(N, N)}^{(N, N)}= \begin{cases}0 & \text { for odd } N, \\ 1 & \text { for even } N.\end{cases}
$$ 

Additionally, the verbatim translation of the property \eqref{two} or equivalently the property \eqref{tensor} that
holds for the Littlewood--Richardson coefficients is also false for the
Kronecker coefficients. One can locate a counterexample
\begin{equation} \label{counter}
\begin{aligned}
&V_{3,3,1,1}^{\otimes 2} - V_{4,4} \otimes V_{2,2,2,2} =V_{8} + 3V_{6, 2} + V_{7, 1} + V_{2, 2, 2, 2} + V_{2, 2, 2, 1, 1} + V_{2, 2, 1, 1, 1, 1}\\
&-V_{1, 1, 1, 1, 1, 1, 1, 1} 
+ 5V_{4, 2, 2} + 3V_{5, 1, 1, 1} + 5V_{5, 2, 1} + 6V_{4, 2, 1, 1} + 5V_{3, 2, 2, 1} + 3V_{4, 1, 1, 1, 1}\\
&+ 5V_{3, 2, 1, 1, 1} + 2V_{3, 1, 1, 1, 1, 1} + V_{6, 1, 1} + 2V_{5, 3} + V_{4, 4} + 5V_{4, 3, 1} + 2V_{3, 3, 2} + 4V_{3, 3, 1, 1}
\end{aligned}
\end{equation}
in the ring of virtual representations of $S_{8}$. The triple of partitions (6,4), (2,2,2,2,2) and (4,3,1,1,1) is a counterexample for $S_{10}$ and there are many more counterexamples for $S_{12}$.

Nevertheless, we conjecture that the verbatim translation of the property \eqref{two} or equivalently the property \eqref{tensor} that holds for the Littlewood--Richardson coefficients is also true for another closely related structure constant---the \emph{reduced Kronecker coefficients}.

\section{Reduced Kronecker Coefficients} \label{Sec-reduced}

For $\lambda$ a partition and $d \geq |\lambda|+\lambda_1$, the ``padded'' partition $\lambda[d]$ is defined as $(d-|\lambda|, \lambda)$, which is a partition of size $d$ with a ``very long top row''. 

It was noticed by Murnaghan in \cite{murnaghan1938analysis} that the sequence $\left\{g_{\lambda[d], \mu[d]}^{\nu[d]}\right\}_{d>>0}$ stabilizes and the stable
value of the sequence was called the \emph{reduced (or stable) Kronecker coefficient} $\bar{g}_{\lambda \mu}^\nu$ associated
with the triple $(\lambda, \mu, \nu)$. Given $\lambda$ and  $\mu$, only finitely many $\bar{g}_{\lambda \mu}^\nu$ are nonzero. Moreover, $\bar{g}_{\lambda \mu}^\nu=0$ unless the Murnaghan--Littlewood inequality holds:
$$
|\nu| \leq |\mu|+|\lambda|, |\mu| \leq |\lambda|+|\nu|, |\lambda| \leq |\mu|+|\nu|.
$$

In contrast to Kronecker coefficients, reduced Kronecker coefficients are defined for any triple of partitions (not necessarily of the same size) and in general, there is no relationship between $\lambda, \mu$, and $\nu$. However, surprisingly, when $|\nu|=|\lambda|+|\mu|$, the reduced Kronecker coefficient $\bar{g}_{\lambda \mu}^\nu$ recovers the Littlewood--Richardson coefficient $c_{\lambda \mu}^{\nu}$!

\begin{thm}[Murnaghan--Littlewood theorem, see \cite{littlewood1958products}] \label{ML}
If $|\nu|=|\lambda|+|\mu|$, then the reduced Kronecker coefficient $\bar{g}_{\lambda \mu}^\nu$ is equal to the Littlewood--Richardson coefficients $c_{\lambda \mu}^{\nu}$:
$\bar{g}_{\lambda \mu}^\nu=c_{\lambda \mu}^{\nu}.$
\end{thm}

Additionally, every Kronecker coefficient is equal to an explicit reduced Kronecker coefficient of not much larger partitions (see \cite[Theorem 1.1]{Ikenmeyer2023all}).

We would like to ask which convexity/concavity property could be satisfied by the reduced Kronecker coefficients.
Whether the verbatim translation of the saturation property \eqref{satu} that holds for the Littlewood--Richardson coefficients is also true for the reduced Kronecker coefficients is a long-standing open problem. It was independently conjectured in 2004 by Kirillov (who called them the \emph{extended Littlewood--Richardson coefficients}) and
Klyachko.

\begin{conj}[Disproved; Kirillov--Klyachko generalized saturation conjecture, see {\cite[Conjecture 2.33]{kirillov2004invitation}} and {\cite[Conjecture 6.2.4]{klyachko2004quantum}}] \label{kconj}
The reduced Kronecker coefficients satisfy the saturation property:
\begin{equation} \label{ksatu}
\bar{g}_{k \lambda, k \mu}^{k \nu} \neq 0 \text { for some } \mathrm{k} \geq 1 \Rightarrow \bar{g}_{\lambda \mu}^{\nu} \neq 0.
\end{equation}
\end{conj}

However, this conjecture is recently refuted in general in \cite{pak2020breaking}:

\begin{thm}[{\cite[Theorem 2]{pak2020breaking}}] \label{break}
For all $k \geq 3$, the triple of partitions $\left(1^{k^2-1}, 1^{k^2-1}, k^{k-1}\right)$ is a counterexample to \cref{kconj}. Moreover, for every partition $\gamma$ s.t. $\gamma_2 \geq 3$, there are infinitely many pairs $(a, b) \in \mathbb{N}^2$ for which the triple of partitions $\left(a^b, a^b, \gamma\right)$ is a counterexample to \cref{kconj}.
\end{thm}

\section{Log-concavity Conjectures of Stable Tensor Product of Irreducible Representations of Symmetric Groups} \label{Sec-Gconj}

One main contribution of this article is the following conjecture.
\begin{conj} \label{Gconj-1}
The reduced Kronecker coefficients satisfy the following inequality: given $\lambda$ and $\mu$, then for all $\nu$, we have
\begin{equation} 
\bar{g}_{\frac{\lambda+\mu}{2} \frac{\lambda+\mu}{2}}^{\nu} \geq \bar{g}_{\lambda \mu}^{\nu},
\end{equation}
provided $\frac{\lambda+\mu}{2}$ is still a partition.
\end{conj}

We tested the above statement for all partitions $\lambda$ and $\mu$ with at most $11$ boxes.

We want an equivalent version of the above conjecture akin to \eqref{tensor}, which can be formulated by using the stable representation category of the symmetric group in \cite{sam2015stability}. 

Consider the natural embedding $S_d \hookrightarrow S_{d+1}$ by permuting the first $d$ natural numbers. Let $S_{\infty}:=\bigcup_{d \geqslant 0} S_d$ be the limit, which is the group of permutations of $\mathbb{N}$ that fix all but finitely many numbers. The group
$S_{\infty}$ has a natural action on $V=\mathbb{C}^{\infty}$ by permuting the basis vectors $\left\{e_i\right\}_{i \in \mathbb{N}}$. Sam and Snowden considered the category $\operatorname{Rep}(S_{\infty})$ of \emph{algebraic representations} of $S_{\infty}$, where a representation of
$S_{\infty}$ is called \emph{algebraic} if it appears as a subquotient of a direct sum of some tensor product of $V$. They proved the following:
\begin{itemize}
    \item $\operatorname{Rep}(S_{\infty})$ is an abelian $\mathbb{C}$-linear symmetric monoidal category but is not semisimple;
    \item Simple objects $V_{\lambda[\infty]}$ in $\operatorname{Rep}(S_{\infty})$ are one-to-one correspondent to the partition $\lambda$ of arbitrary size;
    \item Every object in $\operatorname{Rep}(S_{\infty})$ has finite length;
    \item The structure constants of the Grothendieck ring $
\operatorname{K}(\operatorname{Rep}(S_{\infty}))$
are the reduced Kronecker coefficients: 
$$
\bar{g}_{\lambda, \mu}^{\nu}=\left[V_{\lambda[\infty]} \otimes V_{\mu[\infty]}: V_{\nu[\infty]}\right].
$$
\end{itemize}
Therefore, the category $\operatorname{Rep}(S_{\infty})$ seems to be a natural categorical home of reduced Kronecker coefficients. Let us say that two objects $X$ and $Y$ in $\operatorname{Rep}(S_{\infty})$ satisfy $X \geq Y$ in the Grothendieck ring $\operatorname{K}(\operatorname{Rep}(S_{\infty}))$ if $[X]-[Y]$ in $\operatorname{K}(\operatorname{Rep}(S_{\infty}))$ can be expanded in $V_{\lambda[\infty]}$'s with nonnegative coefficients. Then, our conjecture \ref{Gconj-1} is equivalent to the following

\begin{conj}[Restatement of \cref{Gconj-1}] \label{Gconj-2}
The representation valued function 
\begin{equation}
\begin{aligned}
V: \quad &\mathcal{P} \rightarrow \operatorname{Rep}\left(S_{\infty}\right)\\
&\lambda \longmapsto V_{\lambda[\infty]}
\end{aligned}
\end{equation}
is concave with respect to the natural ordering and tensor products of representations. That is, 
\begin{equation} \label{Tensor}
 V_{\frac{\lambda+\mu}{2}[\infty]}^{\otimes 2} \geq V_{\lambda[\infty]} \otimes V_{\mu[\infty]}\text{ in the Grothendieck ring $\operatorname{K}(\operatorname{Rep}(S_{\infty}))$, }
\end{equation}
provided $\frac{\lambda+\mu}{2}$ is still a partition.
\end{conj}

In this form, the log-concavity of \eqref{tensor} can be seen as a degeneration and a special case of \eqref{Tensor} by virtue of the Murnaghan--Littlewood theorem (\cref{ML}); see also \cite[Section 8.7]{sam2015stability}. Additionally, the conjecture predicts that if we pass to infinity, the mysterious minus sign in \eqref{counter} disappears\footnote{Note that there is no ``sign'' representation in $\operatorname{Rep}(S_{\infty})$.}, and we obtain log-concavity in the limit, which fits well with the conjectures and results in \cites{matherneequivariant,gui2022on}. 

Using \cref{Thm-Schur}, we have the following log-concavity property of the dimensions of representations in \eqref{Tensor}. It greatly generalizes \cite[Theorem 1.1 (1)]{zbMATH07081660}.

\begin{thm} \label{thm-dim}
We have
\begin{equation} \label{dimlog}
    \left(\operatorname{dim}V_{\frac{\lambda+\mu}{2}[d]}\right)^{2} \geq \operatorname{dim}V_{\lambda[d]} \times \operatorname{dim}V_{\mu[d]}
\end{equation}
for $d \geq \operatorname{max}\{|\lambda|+\lambda_1, |\mu|+\mu_1\}$. In another form, we have
\begin{equation} \label{younglog}
\left(f^{\frac{\lambda+\mu}{2}[d]}\right)^{2} \geq f^{\lambda[d]} \times f^{\mu[d]},
\end{equation}
where $f^{\lambda}$ denotes the number of standard Young tableaux of shape $\lambda$.
\end{thm}

\begin{proof}
First, note that $\frac{\lambda[d]+\mu[d]}{2}=\frac{\lambda+\mu}{2}[d]$ under the condition $d \geq \operatorname{max}\{|\lambda|+\lambda_1, |\mu|+\mu_1\}$. Let $s_\lambda$ denote the Schur function of shape $\lambda$. By \cref{Thm-Schur}, we have 
\begin{equation} \label{schurlog}
s_{\frac{\lambda+\mu}{2}[d]}^{2} - s_{\lambda[d]} \times s_{\mu[d]} \geq_s 0,
\end{equation}
which means the left-hand side is a nonnegative linear
combination of Schur functions. Let $\Lambda=\oplus_{n \geq 0} \Lambda_{\mathbb{Q}}^n$ be the algebra of symmetric functions. Then, we have the exponential specialization $\operatorname{ex}_1$, which is an algebra homomorphism $\operatorname{ex}_1: \Lambda \rightarrow \mathbb{Q}$, and
$$
\operatorname{ex}_1\left(s_{\lambda}\right)=\frac{f^{\lambda}}{|\lambda|!}
,$$ 
see, for example, \cite{zbMATH01268810}.
Applying the exponential specialization $\operatorname{ex}_1$ to \eqref{schurlog}), we obtain \eqref{younglog}, which is well-known equivalent to the inequality \eqref{dimlog}. The proof is completed.
\end{proof}

Using the existing combinatorial interpretation of Kronecker coefficients with two two-row partitions, we can prove the following

\begin{thm} \label{thm-tworow}
\cref{Gconj-1}, or equivalently, \cref{Gconj-2}, holds when partitions $\lambda$ and $\mu$ are both one part. Actually, we have the following stronger inequalities:  for all partition $\nu$,
\begin{equation} \label{tworow}
\bar{g}_{(j) (k)}^{\nu} \geq \bar{g}_{(i) (l)}^{\nu},
\end{equation}
whenever $i < j \le k < l$ with $j+k=i+l$.
\end{thm} 

\begin{proof}
First, we have $\bar{g}_{(j), (k)}^{\nu}=g_{(j)[n], (k)[n]}^{\nu[n]}$ and $\bar{g}_{(i), (l)}^{\nu}=g_{(i)[n], (l)[n]}^{\nu[n]}$ for $n$ sufficiently large. It is well known that $g_{\lambda \mu}^{\nu}=0$ when $\lambda$ and $\mu$ are both two-row partitions but $\nu$ has more than 4 parts. Therefore, we assume $\nu[n]=(\nu_1,\nu_2,\nu_3)[n]$. By \cite[Theorem 1]{rosas2001kronecker}, we have 
\begin{equation} \label{comb}
\begin{aligned}
        \bar{g}_{(j)[n],(k)[n]}^{\nu[n]}=&\Gamma\left(\nu_2+\nu_3, \nu_1-\nu_2, \nu_1+\nu_3+1, \nu_2-\nu_3\right)(j, k+1)\\
    &-\Gamma\left(\nu_2+\nu_3, \nu_1-\nu_2, n-\nu_1-\nu_2+2, \nu_2-\nu_3\right)(j, k+1),\\       \bar{g}_{(i)[n],(l)[n]}^{\nu[n]}=&\Gamma\left(\nu_2+\nu_3, \nu_1-\nu_2, \nu_1+\nu_3+1, \nu_2-\nu_3\right)(i, l+1)\\
    &-\Gamma\left(\nu_2+\nu_3, \nu_1-\nu_2, n-\nu_1-\nu_2+2, \nu_2-\nu_3\right)(i, l+1),
    \end{aligned}
\end{equation}
where $\Gamma(a, b, c, d)(x, y):=\left|\left\{(u, v) \in R \cap \mathbf{N}^2:(x, y) \leadsto(u, v)\right\}\right|.$
Here, $R$ is the rectangle with vertices $(a, c),(a+b, c),(a, c+d)$, and $(a+b, c+d)$, and $(x, y) \leadsto(u, v)$ means $(u, v)$ can be reached from $(x, y)$ by moving any number of steps south west or north west. Let $n$ be large enough (that is, we pull up the rectangle $R$ to live very high) such that both minus terms in \eqref{comb} are 0. Then, it is clear from the definition that 
\begin{align*} 
& \Gamma\left(\nu_2+\nu_3,\nu_1-\nu_2,\nu_1+\nu_3+1,\nu_2-\nu_3\right)(j,k+1) \\ & \quad \geq \Gamma\left(\nu_2+\nu_3,\nu_1-\nu_2,\nu_1+\nu_3+1,\nu_2-\nu_3\right)(i,l+1),
\end{align*}
since $(i, l+1)$ is in the northwest of $(j, k+1)$. Therefore, we have
$$\bar{g}_{(j) (k)}^{\nu}=g_{(j)[n], (k)[n]}^{\nu[n]} \geq g_{(i)[n], (l)[n]}^{\nu[n]}=\bar{g}_{(i) (l)}^{\nu}.$$
\end{proof}

Let us now discuss a conjecture related to \cref{Gconj-2}. For two partitions $\lambda$ and $\mu$, let $\lambda \cup \mu=\left(\nu_1, \nu_2, \nu_3, \ldots\right)$ be
the partition obtained by rearranging all parts of $\lambda$ and $\mu$ in the weakly decreasing order. Let $\operatorname{sort}_1(\lambda, \mu):=\left(\nu_1, \nu_3, \nu_5, \ldots\right)$ and $\operatorname{sort}_2(\lambda, \mu):=\left(\nu_2, \nu_4, \nu_6, \ldots\right)$. Then, we have the following conjecture, which generalizes the conjecture of Fomin, Fulton, Li, and Poon in \cite[Conjecture 2.7]{fomin2005eigenvalues}.

\begin{conj} \label{Gconj-3}
For two partitions $\lambda$ and $\mu$, we have 
$$
 V_{\text {sort}_1(\lambda, \mu)[\infty]} \otimes V_{\mathrm{sort}_2(\lambda, \mu)[\infty]} \geq V_{\lambda[\infty]} \otimes V_{\mu[\infty]} \text{ in the Grothendieck ring $\operatorname{K}(\operatorname{Rep}(S_{\infty}))$}.$$
\end{conj}

As observed in \cite{lam2007schur}, \cref{Gconj-3} is related to \cref{Gconj-2} by conjugating the shapes. However, since we have to add a ``very long top row'', \cref{Gconj-3} can not be directly deduced from (even a stronger version of) \cref{Gconj-2}, unlike the case in \cite{lam2007schur}. Indeed, we can not even prove the following inequality for dimensions of representations 
\begin{equation} 
f^{\text {sort}_1(\lambda, \mu)[d]} \times f^{\text {sort}_2(\lambda, \mu)[d]} \geq f^{\lambda[d]} \times f^{\mu[d]} \text{ for } d \geq \operatorname{max}\{|\lambda|+\lambda_1, |\mu|+\mu_1\}
\end{equation}
by directly using results in \cite{lam2007schur}. Nevertheless, using the existing combinatorial interpretation of Kronecker coefficients with two hook-shape partitions, we have the following

\begin{thm} \label{thm-hook}
\cref{Gconj-3} holds when partitions $\lambda$ and $\mu$ are both one column. 
Actually, we have the following stronger inequalities
\begin{equation} \label{hook}
\bar{g}_{(1^j) (1^k)}^{\nu} \geq \bar{g}_{(1^i) (1^l)}^{\nu}, \quad \text{ for all partition } \nu,
\end{equation} 
whenever $i < j \le k < l$ with $j+k=i+l$.
\end{thm}

\begin{proof}
First, we have $\bar{g}_{(1^j) (1^k)}^{\nu}=g_{(1^j)[n], (1^k)[n]}^{\nu[n]}$ and $\bar{g}_{(1^i) (1^l)}^{\nu}=g_{(1^i)[n], (1^l)[n]}^{\nu[n]}$ for $n$ sufficiently large. By \cite[Theorem 3]{rosas2001kronecker}, the only possible values for these Kronecker
coefficients are $0$, $1$ or $2$. We use the following notation
$$
((P))= \begin{cases}1, & \text { if proposition } P \text { is true, } \\ 0, & \text { otherwise. }\end{cases}
$$

We have the following 4 cases:\\
1. If $\nu[n]$ is one-row (i.e.\ $\nu=\emptyset$), then $g_{(1^j)[n], (1^k)[n]}^{\nu[n]}=\delta_{j, k}$, $g_{(1^i)[n], (1^l)[n]}^{\nu[n]}=\delta_{i, l}$.\\
2. If $\nu[n]$ is not contained in a double hook, then $g_{(1^j)[n], (1^k)[n]}^{\nu[n]}=g_{(1^i)[n], (1^l)[n]}^{\nu[n]}=0$.\\
3. Let $\nu[n]=\left(1^{d_1} 2^{d_2} n_3 n_4\right)$ be a double hook. Let $x=2 d_2+d_1$. Then,
$$
\begin{aligned}
g_{(1^j)[n], (1^k)[n]}^{\nu[n]}=&\left(\left(n_3-1 \leq \frac{j+k-x}{2} \leq n_4\right)\right)\left(\left(|k-j| \leq d_1\right)\right) \\
&+\left(\left(n_3 \leq \frac{j+k-x+1}{2} \leq n_4\right)\right)\left(\left(|k-j| \leq d_1+1\right)\right)\\
g_{(1^i)[n], (1^l)[n]}^{\nu[n]}=&\left(\left(n_3-1 \leq \frac{i+l-x}{2} \leq n_4\right)\right)\left(\left(|l-i| \leq d_1\right)\right) \\
&+\left(\left(n_3 \leq \frac{i+l-x+1}{2} \leq n_4\right)\right)\left(\left(|l-i| \leq d_1+1\right)\right).
\end{aligned}
$$
Note that if $n_4=0$, then we shall rewrite $\nu[n]=\left(1^{d_1} 2^{d_2-1} 2 n_3\right)$.\\
4. Let $\nu[n]=\left(1^d w\right)$ be a hook shape. Let $n$ be sufficiently large. Then,
$$
\begin{aligned}
g_{(1^j)[n], (1^k)[n]}^{\nu[n]}&=((j \leq d+k))((d \leq j+k))((k \leq j+d)),\\
g_{(1^i)[n], (1^l)[n]}^{\nu[n]}&=((i \leq d+l))((d \leq i+l))((l \leq i+d)).
\end{aligned}
$$

It is not hard to see that \eqref{hook} holds in any case, which easily implies \cref{Gconj-3} when partitions $\lambda$ and $\mu$ are both one column. The proof is completed.
\end{proof}

We note that inequality \eqref{tworow} and inequality \eqref{hook} show a beautiful symmetry that is not transparent if we do not remove the ``very long top rows''.

\cref{Gconj-3} is useful. For example, it implies that the intersection cohomology of the symmetric reciprocal plane $X_n$ (see in \cite[Theorem 1.2]{proudfoot2016intersection}) is equivariant log-concave at degree $i$ as a graded representation of the symmetric group $S_n$ for $n$ large enough.

We note that Conjectures \ref{Gconj-2} and \ref{Gconj-3} can also be formalized as \emph{Schur positivity conjectures} using the \emph{Schur functions} and the \emph{Kronecker (internal) product} between them, provided that the related partitions all have “long enough top rows”. These conjectures could also be formalized using the Deligne category or partition algebra, which might shed some light on these conjectures. However, due to the lack of general knowledge of the (reduced) Kronecker coefficients, proving these conjectures, in general, seems to be beyond the reach of existing technology.

\textcolor{red}{As informed by Mike Zabrocki, Conjectures \ref{Gconj-1} is false. The first counterexamples that he found are partitions $\lambda$ and $\mu$ with $12$ boxes: all of
${\overline g}_{4422, 4422}^{\lambda}$
for $\lambda$ in $\{(1^{16}), (1^{15}), (2, 1^{14}), (3, 1^{14}), (2, 2, 1^{13})\}$
are equal to $0$, while
${\overline g}_{5511, 3333}^{\lambda}$
are all greater than $0$ (in fact they are $1,2,3,1,1$, respectively). Those are the only examples where the conjecture fails for two partitions of $12$. The fact that Conjecture \ref{Gconj-1} first fails for $|\lambda|=|\mu|=12$ is somewhat surprising. It is interesting to see whether one can revise or refine the conjecture.}

\acknowledgements{The author is very grateful to Mike Zabrocki for finding and telling him the counterexamples. The author would like to thank Brendon Rhoades, Peng Shan, and Arthur L. B. Yang for helpful discussions and is especially grateful to Matthew H.Y. Xie for the help with the computer computations of tensor product multiplicities. The author also thanks the anonymous referees for the comments and suggestions.}

\printbibliography

\end{document}